\documentclass{gtpart}     
\usepackage{graphicx}  
\usepackage{latexsym}
\title[Infinite generation for right-angled buildings]{Infinite generation of non-cocompact lattices \\ on right-angled buildings}

\author[A Thomas]{Anne Thomas}
\givenname{Anne}
\surname{Thomas}
\address{School of Mathematics and Statistics \\University of Sydney}
\email{athomas@maths.usyd.edu.au}
\urladdr{}

\author[K Wortman]{Kevin Wortman}
\givenname{Kevin}
\surname{Wortman}
\address{Department of Mathematics\\ University of Utah}
\email{wortman@math.utah.edu}
\urladdr{}

\arxivreference{1009.4235}
\arxivpassword{abitf}

\volumenumber{}
\issuenumber{}
\publicationyear{}
\papernumber{}
\startpage{}
\endpage{}
\doi{}
\MR{}
\Zbl{}
\received{}
\revised{}
\accepted{}
\published{}
\publishedonline{}
\proposed{}
\seconded{}
\corresponding{}
\editor{}
\version{}

\newtheorem{theorem}{Theorem}
\newtheorem{lemma}{Lemma}          
\newtheorem{proposition}{Proposition}
\newtheorem{corollary}{Corollary}

\makeatletter
\let\c@lemma=\c@theorem
\makeatother

\makeatletter
\let\c@proposition=\c@theorem
\makeatother

\makeatletter
\let\c@corollary=\c@theorem
\makeatother

\def\T{{\mathcal T}}

\def\rperp{\langle r^\perp \rangle}
\def\sperp{\langle s^\perp \rangle}
\def\tperp{\langle t^\perp \rangle}
\def\siperp{\langle s_i^\perp \rangle}

\DeclareMathOperator\Aut{Aut}
\DeclareMathOperator\Stab{Stab}
\DeclareMathOperator\Ch{Ch}

\def\polhk#1{\setbox0=\hbox{#1}{\ooalign{\hidewidth
    \lower1.0ex\hbox{$\,\lhook$}\hidewidth\crcr\unhbox0}}}
\newcommand{\Swiatkowski}{\'Swi{\polhk{a}}tkowski}

\begin{document}

\begin{abstract}    

Let $\Gamma$ be a non-cocompact lattice on a locally finite regular right-angled building $X$.  We prove that if $\Gamma$ has a strict fundamental domain then $\Gamma$ is not finitely generated.  We use the separation properties of subcomplexes of $X$ called tree-walls.

\end{abstract}

\maketitle


 Tree lattices have been well-studied (see~\cite{BL}).  Less understood are lattices on higher-dimensional CAT(0) complexes.  In this paper, we consider lattices on $X$ a locally finite, regular right-angled building (see Davis \cite{D} and \fullref{s:rabs} below).  Examples of such $X$ include products of locally finite regular or biregular trees, or Bourdon's building $I_{p,q}$  \cite{B}, which has apartments hyperbolic planes tesselated by right-angled $p$--gons and all vertex links the complete bipartite graph $K_{q,q}$.

Let $G$ be a closed, cocompact group of type-preserving automorphisms of $X$, equipped with the compact-open topology, and let $\Gamma$ be a lattice in $G$.  That is, $\Gamma$ is discrete and the series $\sum |\Stab_\Gamma(\phi)|^{-1}$
converges, where the sum is over the set of chambers $\phi$ of a fundamental domain for $\Gamma$.  The lattice $\Gamma$ is cocompact in $G$ if and only if the quotient $\Gamma \backslash X$ is compact.

If there is a subcomplex $Y \subset X$ containing exactly one point from each $\Gamma$--orbit on $X$, then $Y$ is called a \emph{strict fundamental domain} for $\Gamma$.  Equivalently, $\Gamma$ has a strict fundamental domain if $\Gamma \backslash X$ may be embedded in $X$.

Any cocompact lattice in $G$ is finitely generated.  We prove:

\begin{theorem}\label{t:strict implies not fg}  Let $\Gamma$ be a non-cocompact lattice in $G$.  If $\Gamma$ has a strict fundamental domain, then $\Gamma$ is \emph{not} finitely generated.
\end{theorem}

We note that \fullref{t:strict implies not fg} contrasts with the finite generation of lattices on many buildings whose chambers are simplices.  Results of, for example, Ballmann--\Swiatkowski\ \cite{BS},  Dymara--Januszkiewicz \cite{DJ}, and Zuk \cite{Z}, establish that all lattices on many such buildings have Kazhdan's Property (T).  Hence by a well-known result due to Kazhdan \cite{K}, these lattices are finitely generated.

Our proof of \fullref{t:strict implies not fg}, in \fullref{s:proof} below,
uses the separation properties of subcomplexes of $X$ which we call \emph{tree-walls}.  These generalize the tree-walls (in French, \emph{arbre-murs}) of $I_{p,q}$, which were introduced by Bourdon in  \cite{B}.  We define  tree-walls and establish their properties in \fullref{s:tree-walls} below.

The following examples of non-cocompact lattices on right-angled buildings are known to us.
\begin{enumerate}
\item For $i = 1,2$, let $G_i$ be a rank one Lie group over a nonarchimedean locally compact field whose Bruhat--Tits building is the locally finite regular or biregular tree $T_i$.  Then any irreducible lattice in $G = G_1 \times G_2$ is finitely generated (Raghunathan \cite{Ra}).  Hence by \fullref{t:strict implies not fg} above, such lattices on $X = T_1 \times T_2$ cannot have strict fundamental domain.
\item Let $\Lambda$ be a minimal Kac--Moody group over a finite field $\mathbb{F}_q$ with right-angled Weyl group $W$.  Then $\Lambda$ has locally finite, regular right-angled twin buildings $X_+ \cong X_-$, and $\Lambda$ acts diagonally on the product $X_+ \times X_-$.  For $q$ large enough:
\begin{enumerate}
\item
By Theorem  0.2 of Carbone--Garland \cite{CG} or Theorem 1(i) of R\'emy \cite{R}, the stabilizer in $\Lambda$ of a point in $X_-$ is a non-cocompact lattice in $\Aut(X_+)$.  Any such lattice is contained in a negative maximal spherical parabolic subgroup of $\Lambda$, which has strict fundamental domain a sector in $X_+$, and so any such lattice has strict fundamental domain.  \item By Theorem 1(ii) of R\'emy \cite{R}, the group $\Lambda$ is itself a non-cocompact lattice in $\Aut(X_+) \times \Aut(X_-)$.  Since $\Lambda$ is finitely generated, \fullref{t:strict implies not fg} above implies that $\Lambda$ does not have  strict fundamental domain in $X = X_+ \times X_-$.
\item By Section 7.3 of Gramlich--Horn--M\"uhlherr \cite{GHM}, the fixed set $G_\theta$ of certain involutions $\theta$ of $\Lambda$ is a lattice in $\Aut(X_+)$, which is sometimes cocompact and sometimes non-cocompact.  Moreover, by \cite[Remark 7.13]{GHM}, there exists $\theta$ such that $G_\theta$ is not finitely generated.
\end{enumerate}
\item In \cite{T}, the first author constructed a functor from graphs of groups to complexes of groups, which extends the corresponding tree lattice to a lattice in $\Aut(X)$ where $X$ is a regular right-angled building.  The resulting lattice in $\Aut(X)$ has strict fundamental domain if and only if the original tree lattice has strict fundamental domain.
\end{enumerate}

\subsection*{Acknowledgements}

The first author was supported in part by NSF Grant No. DMS-0805206 and in part by EPSRC Grant No. EP/D073626/2, and is currently supported by ARC Grant No. DP110100440.  The second author is supported in part by NSF Grant No. DMS-0905891.  
We thank Martin Bridson and Pierre-Emmanuel Caprace for helpful conversations.

\section{Right-angled buildings}\label{s:rabs}

In this section we recall the basic definitions and some examples for right-angled buildings.  We mostly follow Davis \cite{D}, in particular Section 12.2 and Example 18.1.10.  See also \cite[Sections 1.2--1.4]{KT}.

Let $(W,S)$ be a right-angled Coxeter system.  That is, $$W = \langle S \mid (st)^{m_{st}} = 1\rangle$$ where $m_{ss} = 1$ for all $s \in S$, and $m_{st} \in \{ 2,\infty \}$ for all $s, t \in S$ with $s \neq t$.
We will discuss the following examples:
\begin{itemize}
\item $W_1 = \langle s, t \mid s^2 = t^2 = 1\rangle \cong D_\infty$, the infinite dihedral group;
\item $W_2 = \langle r, s, t \mid r^2 = s^2 = t^2 = (rs)^2 = 1  \rangle \cong (C_2 \times C_2)* C_2$, where $C_2$ is the cyclic group of order $2$; and
\item The Coxeter group $W_3$ generated by the set of reflections $S$ in the sides of a right-angled hyperbolic $p$--gon, $p \geq 5$.  That is, \[ W_3 = \langle s_1, \ldots, s_p \mid s_i^2 = (s_i s_{i+1})^2 = 1  \rangle \] with cyclic indexing.
\end{itemize}

Fix $(q_s)_{s \in S}$ a family of integers with $q_s \geq 2$.  Given any family of groups $(H_s)_{s \in S}$ with $|H_s| = q_s$, let $H$ be the quotient of the free product of the $(H_s)_{s \in S}$ by the normal subgroup generated by the commutators $\left\{ [h_s,h_t] : h_s \in H_s, h_t \in H_t, m_{st} = 2\right\}$.  

Now let $X$ be the piecewise Euclidean CAT(0) geometric realization of the chamber system $\Phi=\Phi\left(H,\{1\},(H_s)_{s \in S}\right)$.  Then $X$ is a locally finite, regular right-angled building, with chamber set $\Ch(X)$ in bijection with the elements of the group $H$.  
Let $\delta_W\co\Ch(X) \times \Ch(X) \to W$ be the $W$--valued distance function and let $l_S\co W \to \mathbb{N}$ be word length with respect to the generating set $S$.   Denote by $d_W\co \Ch(X) \times \Ch(X) \to \mathbb{N}$ the \emph{gallery distance} $l_S \circ \delta_W$.  That is, for two chambers $\phi$ and $\phi'$ of $X$, $d_W(\phi,\phi')$ is the length of a minimal gallery from  $\phi$ to $\phi'$.

Suppose that $\phi$ and $\phi'$ are $s$--adjacent chambers, for some $s \in S$.  That is, $\delta_W(\phi,\phi') = s$.  The intersection $\phi \cap \phi'$ is called an \emph{$s$--panel}.  By definition, since $X$ is regular, each $s$--panel is contained in $q_s$ distinct chambers.  For distinct $s, t \in S$, the $s$--panel and $t$--panel of any chamber $\phi$ of $X$ have nonempty intersection if and only if $m_{st} = 2$.  Each $s$--panel of $X$ is reduced to a vertex if and only if $m_{st} = \infty$ for all $t \in S - \{ s \}$.

For the examples $W_1$, $W_2$, and $W_3$ above, respectively:
\begin{itemize}
\item The building $X_1$ is a tree with each chamber an edge, each $s$--panel a vertex of valence $q_s$, and each $t$--panel a vertex of valence $q_t$.  That is, $X_1$ is the $(q_s,q_t)$--biregular tree.  The apartments of $X_1$ are bi-infinite rays in this tree.
\item The building $X_2$ has chambers and apartments as shown in \fullref{f:RAB_example2} below.  The $r$-- and $s$--panels are $1$--dimensional and the $t$--panels are vertices.
\begin{figure}[ht]
\begin{center}
\scalebox{0.5}{\includegraphics{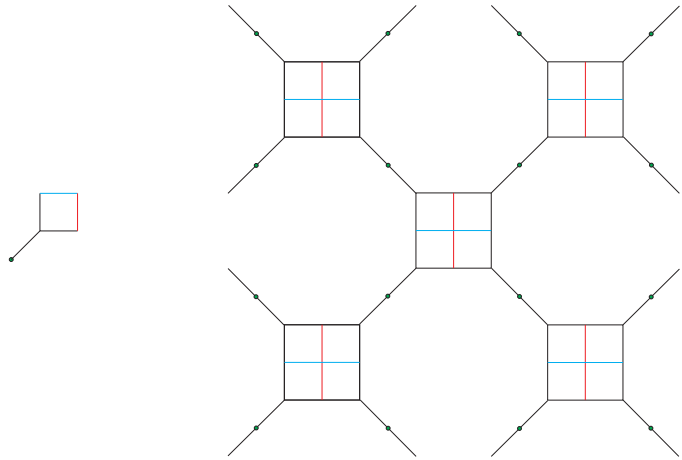}}
\caption{A chamber (on the left) and part of an apartment (on the right) for the building $X_2$.}
\label{f:RAB_example2}
\end{center}
\end{figure}

\item The building $X_3$ has chambers $p$--gons and $s$--panels the edges of these $p$--gons.  If $q_s = q \geq 2$ for all $s \in S$, then each $s$--panel is contained in $q$ chambers, and $X_3$, equipped with the obvious piecewise hyperbolic metric, is Bourdon's building $I_{p,q}$.
\end{itemize}

\section{Tree-walls}\label{s:tree-walls}

We now generalize the notion of tree-wall due to Bourdon \cite{B}.  We will use basic facts about buildings, found in, for example, Davis \cite{D}.  Our main results concerning tree-walls are \fullref{c:tree wall trichotomy} below, which describes three possibilities for tree-walls, and \fullref{p:tree-wall} below, which generalizes the separation property 2.4.A(ii) of \cite{B}.

Let $X$ be as in \fullref{s:rabs} above and let $s \in S$.  As in \cite[Section 2.4.A]{B}, we define two $s$--panels of $X$ to be \emph{equivalent} if they are contained in a common wall of type $s$ in some apartment of $X$.  A \emph{tree-wall of type $s$} is then an equivalence class under this relation.  
We note that in order for walls and thus tree-walls to have a well-defined type, it is necessary only that all finite $m_{st}$, for $s \neq t$, be even.  Tree-walls could thus be defined for buildings of type any even Coxeter system, and they would have properties similar to those below.  We will however only explicitly consider the right-angled case.

Let $\T$ be a tree-wall of $X$, of type $s$.  We define a chamber $\phi$ of $X$ to be \emph{epicormic at $\T$} if the $s$--panel of $\phi$ is contained in $\T$, and we say that a gallery $\alpha = (\phi_0, \ldots, \phi_n)$ \emph{crosses $\T$} if, for some $0 \leq i < n$, the chambers $\phi_i$ and $\phi_{i + 1}$ are epicormic at $\T$.  

By the definition of tree-wall, if $\phi \in \Ch(X)$ is epicormic at $\T$ and $\phi' \in \Ch(X)$ is $t$--adjacent to $\phi$ with $t \neq s$, then $\phi'$ is epicormic at $\T$ if and only if $m_{st} = 2$.
Let $s^\perp := \{ t \in S \mid m_{st} = 2\}$ and denote by $\sperp$ the subgroup of $W$ generated by the elements of $s^\perp$.    If $s^\perp$ is empty then by convention, $\sperp$ is trivial.  For the examples in \fullref{s:rabs} above:
\begin{itemize}
\item in $W_1$, both $\sperp$ and $\tperp$ are trivial;
\item in $W_2$, $\rperp = \langle s \rangle \cong C_2$ and $\sperp = \langle r \rangle \cong C_2$, while $\tperp$ is trivial; and
\item in $W_3$, $\siperp = \langle s_{i-1}, s_{i+1} \rangle \cong D_\infty$ for each $1 \leq i \leq p$.
\end{itemize}

\begin{lemma}\label{l:sperp epicormic}  Let $\T$ be a tree-wall of $X$ of type $s$.
 Let $\phi$ be a chamber which is epicormic at $\T$ and let $A$ be any apartment containing $\phi$.
\begin{enumerate}
\item\label{i:wall separates} The intersection $\T \cap A$ is a wall of $A$, hence separates $A$.
\item\label{i:word in s perp} There is a bijection between the elements of the group $\sperp$ and the set of chambers of $A$ which are epicormic at $\T$ and in the same component of $A  - \T \cap A$ as $\phi$.
\end{enumerate}
\end{lemma}

\begin{proof}
 Part \eqref{i:wall separates} is immediate from the definition of tree-wall.  For Part \eqref{i:word in s perp}, let $w \in \sperp$ and let $\psi=\psi_w$ be the unique chamber of $A$ such that $\delta_W(\phi,\psi) = w$.  We claim that $\psi$ is epicormic at $\T$ and in the same component of $A - \T \cap A$ as $\phi$.

For this, let $s_1 \cdots s_n$ be a reduced expression for $w$ and let $\alpha = (\phi_0,\ldots,\phi_n)$ be the minimal gallery from $\phi = \phi_0$ to $\psi=\phi_n$ of type $(s_1,\ldots,s_n)$.  Since $w$ is in $\sperp$, we have $m_{s_i s} = 2$ for $1 \leq i \leq n$. Hence by induction each $\phi_i$ is epicormic at $\T$, and so $\psi = \phi_n$ is epicormic at $\T$.  Moreover, since none of the $s_i$ are equal to $s$, the gallery $\alpha$ does not cross $\T$.  Thus $\psi=\psi_w$ is in the same component of $A - \T \cap A$ as $\phi$.  

It follows that $w \mapsto \psi_w$ is a well-defined, injective map from $\sperp$ to the set of chambers of $A$ which are epicormic at $\T$ and in the same component of $A - \T \cap A$ as $\phi$.  To complete the proof, we will show that this map is surjective.  So let $\psi$ be a chamber of $A$ which is epicormic at $\T$ and in the same component of $A - \T \cap A$ as $\phi$, and let $w = \delta_W(\phi,\psi)$.

If $\sperp$ is trivial then $\psi = \phi$ and $w = 1$, and we are done.  Next suppose that the chambers $\phi$ and $\psi$ are $t$--adjacent, for some $t \in S$.  Since both $\phi$ and $\psi$ are epicormic at $\T$, either $t = s$ or $m_{st} = 2$.  But $\psi$ is in the same component of $A - \T \cap A$ as $\phi$, so $t \neq s$, hence $w = t$ is in $\sperp$ as required.  If $\sperp$ is finite, then finitely many applications of this argument will finish the proof.  If $\sperp$ is infinite, we have established the base case of an induction on $n=l_S(w)$.  

For the inductive step, let $s_1 \cdots s_n$ be a reduced expression for $w$  and let $\alpha = (\phi_0,\ldots,\phi_n)$ be the minimal gallery from $\phi = \phi_0$ to $\psi=\phi_n$ of type $(s_1,\ldots,s_n)$.  Since $\phi$ and $\psi$ are in the same component of $A - \T \cap A$ and $\alpha$ is minimal, the gallery $\alpha$ does not cross $\T$.  
We claim that $s_n$ is in $s^\perp$.  First note that $s_n \neq s$ since $\alpha$ does not cross $\T$ and $\psi = \phi_n$ is epicormic at $\T$.  Now denote by $\T_n$ the tree-wall of $X$ containing the $s_n$--panel $\phi_{n-1} \cap \phi_n$.    Since $\alpha$ is minimal and crosses $\T_n$, the chambers $\phi = \phi_0$ and $\psi = \phi_n$ are separated by the wall $\T_n \cap A$.  Thus the $s$--panel of $\phi$ and the $s$--panel of $\psi$ are separated by $\T_n \cap A$.  As the $s$--panels of both $\phi$ and $\psi$ are in the wall $\T \cap A$, it follows that the walls $\T_n \cap A$ and $\T \cap A$ intersect.  Hence $m_{s_n s} = 2$, as claimed.

Now let $w' = ws_n = s_1 \cdots s_{n-1}$ and let $\psi'$ be the unique chamber of $A$ such that $\delta_W(\phi,\psi') = w'$.  Since $s_n$ is in $s^\perp$ and $\psi'$ is $s_n$--adjacent to $\psi$, the chamber $\psi'$ is epicormic at $\T$ and in the same component of $A - \T \cap A$ as $\phi$.  Moreover $s_1 \cdots s_{n-1}$ is a reduced expression for $w'$, so $l_S(w') = n -1$.  Hence by the inductive assumption, $w'$ is in $\sperp$.  Therefore  $w = w's_n$ is in $\sperp$, which completes the proof.
\end{proof}

\begin{corollary}\label{c:tree wall trichotomy}  The following possibilities for tree-walls in $X$ may occur.
\begin{enumerate}
\item\label{i:reduced to vertex} Every tree-wall of type $s$ is reduced to a vertex if and only if $\sperp$ is trivial.  
\item\label{i:finite} Every tree-wall of type $s$ is finite but not reduced to a vertex if and only if $\sperp$ is finite but nontrivial. 
\item\label{i:infinite} Every tree-wall of type $s$ is infinite if and only if $\sperp$ is infinite. 
\end{enumerate}
\end{corollary}

\begin{proof}  Let $\T$, $\phi$, and $A$ be as in \fullref{l:sperp epicormic} above.
 The set of $s$--panels in the wall $\T \cap A$ is in bijection with the set of chambers of $A$ which are epicormic at $\T$ and in the same component of $A - \T \cap A$ as $\phi$.
\end{proof}

\noindent For the examples in \fullref{s:rabs} above:
\begin{itemize}
 \item in $X_1$,  every tree-wall of type $s$ and of type $t$ is a vertex;
 \item in $X_2$, the tree-walls of types both $r$ and $s$ are finite and $1$--dimensional, while every tree-wall of type $t$ is a vertex; and
\item in $X_3$, all tree-walls are infinite, and are $1$--dimensional.
\end{itemize}

\begin{corollary}\label{c:inverse image} Let $\T$, $\phi$, and $A$ be as in \fullref{l:sperp epicormic} above and let
\[ \rho = \rho_{\phi,A}\co X \to A \]
be the retraction onto $A$ centered at $\phi$.  Then $\rho^{-1}(\T \cap A) = \T$.
\end{corollary}

\begin{proof} Let $\psi$ be any chamber of $A$ which is  epicormic at $\T$ and is in the same component of $A - \T \cap A$ as $\phi$.  Then by the proof of  \fullref{l:sperp epicormic} above, $w := \delta_W(\phi,\psi)$ is in $\sperp$.  
Let $\psi'$ be a chamber in the preimage $\rho^{-1}(\psi)$ and let $A'$ be an apartment containing both $\phi$ and $\psi'$.  Since the retraction $\rho$ preserves $W$--distances from $\phi$, we have that $\delta_W(\phi,\psi') = w$ is in $\sperp$.  Again by the proof of \fullref{l:sperp epicormic}, it follows that the chamber $\psi'$ is epicormic at $\T$.  But the image under $\rho$ of the $s$--panel of $\psi'$ is the $s$--panel of $\psi$.  Thus $\rho^{-1}(\T \cap A) = \T$, as required.
\end{proof}

 \begin{lemma}\label{l:galleries crossing} Let $\T$ be a tree-wall and let $\phi$ and $\phi'$ be two chambers of $X$.  Let $\alpha$ be a minimal gallery from $\phi$ to $\phi'$ and let $\beta$ be any gallery from $\phi$ to $\phi'$. If $\alpha$ crosses $\T$ then $\beta$ crosses $\T$.
\end{lemma}

\begin{proof}  Suppose that $\alpha$ crosses $\T$.  Since $\alpha$ is minimal, there is an apartment $A$ of $X$ which contains $\alpha$, and hence the wall $\T \cap A$ separates $\phi$ from $\phi'$.  Choose a chamber $\phi_0$ of $A$ which is epicormic at $\T$ and consider the retraction $ \rho = \rho_{\phi_0,A}$ onto $A$ centered at $\phi_0$.
Since $\phi$ and $\phi'$ are in $A$, $\rho$ fixes $\phi$ and $\phi'$.  Hence $\rho(\beta)$ is a gallery in $A$ from $\phi$ to $\phi'$, and so $\rho(\beta)$ crosses $\T\cap A$.  By \fullref{c:inverse image} above, $\rho^{-1}(\T \cap A) = \T$.  Therefore  $\beta$ crosses $\T$. \end{proof}

\begin{proposition}\label{p:tree-wall} Let $\T$ be a tree-wall of type $s$.  Then $\T$ separates $X$ into $q_s$ gallery-connected components.
\end{proposition}

\begin{proof} Fix an $s$--panel in $\T$ and let  $\phi_1,\ldots,\phi_{q_s}$ be the $q_s$ chambers containing this panel. Then  for all $1 \leq i < j \leq q_s$, the minimal gallery from $\phi_i$ to $\phi_j$ is just $(\phi_i,\phi_j)$, and hence crosses $\T$.  Thus by \fullref{l:galleries crossing} above, any  gallery from $\phi_i$ to $\phi_j$ crosses $\T$.  So the $q_s$ chambers $\phi_1,\ldots,\phi_{q_s}$ lie in $q_s$ distinct components of $X - \T$.

To complete the proof, we show that $\T$ separates $X$ into at most $q_s$ components.  Let $\phi$ be any chamber of $X$.  Then among the chambers $\phi_1,\ldots,\phi_{q_s}$, there is a unique chamber, say $\phi_1$, at minimal gallery distance from $\phi$.   It suffices to show that $\phi$ and $\phi_1$ are in the same component of $X - \T$.

Let $\alpha$ be a minimal gallery from $\phi$ to $\phi_1$ and let $A$ be an apartment
containing $\alpha$.  Then there is a unique chamber of $A$ which is $s$--adjacent to $\phi_1$.  Hence $A$ contains $\phi_i$ for some $i > 1$, and  the wall $\T \cap A$ separates $\phi_1$ from $\phi_i$.  Since $\alpha$ is minimal and $d_W(\phi,\phi_1) < d_W(\phi,\phi_i)$, the Exchange Condition (see \cite[page 35]{D}) implies that a minimal gallery from $\phi$ to $\phi_i$ may be obtained by concatenating $\alpha$ with the gallery $(\phi_1,\phi_i)$. Since a minimal gallery can cross $\T \cap A$ at most once, $\alpha$ does not cross $\T \cap A$.  Thus  $\phi$ and $\phi_1$ are in the same component of $X - \T$, as required.
\end{proof}

\section{Proof of Theorem}\label{s:proof}

Let $G$ be as in the introduction and let $\Gamma$ be a non-cocompact lattice in $G$ with strict fundamental domain.  Fix a chamber $\phi_0$ of $X$.  For each integer $n \geq 0$ define \[D(n):=\{\,\phi \in \Ch(X) \mid d_W(\phi,\Gamma \phi_0) \leq n \,\}.\]  Then $D(0)=\Gamma \phi_0$, and for every $n > 0$ every connected  component of $D(n)$ contains a chamber in $\Gamma \phi_0$.  To prove \fullref{t:strict implies not fg}, we will show that there is no $n > 0$ such that $D(n)$ is connected.

Let $Y$ be a strict fundamental domain for $\Gamma$ which contains $\phi_0$.  For each chamber $\phi$ of $X$, denote by $\phi_Y$ the representative of $\phi$ in $Y$.

\begin{lemma}\label{l:projection preserves adjacency} Let $\phi$ and $\phi'$ be $t$--adjacent chambers in $X$, for $t \in S$.  Then either $\phi_Y = \phi'_Y$, or $\phi_Y$ and $\phi'_Y$ are $t$--adjacent.
\end{lemma}

\begin{proof}
 It suffices to show that the $t$--panel of $\phi_Y$ is the $t$--panel of $\phi'_Y$.  Since $Y$ is a subcomplex of $X$, the $t$--panel of $\phi_Y$ is contained in $Y$.   By definition of a strict fundamental domain, there is exactly one representative in $Y$ of the $t$--panel of $\phi$.  Hence the unique representative in $Y$ of the $t$--panel of $\phi$ is the $t$--panel of $\phi_Y$.  Similarly, the unique representative in $Y$ of the $t$--panel of $\phi'$ is the $t$--panel of $\phi'_Y$.  But $\phi$ and $\phi'$ are $t$--adjacent, hence have the same $t$--panel, and so it follows that $\phi_Y$ and $\phi'_Y$ have the same $t$--panel.
\end{proof}

\begin{corollary}\label{c:gallery-connected} The fundamental domain $Y$ is gallery-connected.  \end{corollary}

\begin{lemma}\label{l:transverse gallery} For all $n > 0$, the fundamental domain $Y$ contains a pair of adjacent chambers $\phi_n$ and $\phi'_n$ such that, if $\T_n$ denotes the tree-wall separating $\phi_n$ from $\phi'_n$:
\begin{enumerate}
\item\label{i:phi_0 and phi_n} the chambers $\phi_0$ and $\phi_n$ are in the same gallery-connected component of $Y - \T_n \cap Y$;
\item\label{i:distance to T_n} $\min \{ d_W(\phi_0,\phi) \mid \phi \in \Ch(X) \mbox{ is epicormic at $\T_n$} \} > n$; and
\item\label{i:stabilizer} there is a $\gamma \in \Stab_\Gamma(\phi'_n)$ which does not fix $\phi_n$.
\end{enumerate}
\end{lemma}

\begin{proof} Fix $n > 0$.  Since $\Gamma$ is not cocompact, $Y$ is not compact.  Thus there exists a tree-wall $\T_n$ with $\T_n \cap Y$ nonempty such that for every $\phi \in \Ch(X)$ which is epicormic at $\T_n$, $d_W(\phi_0,\phi) > n$.  Let $s_n$ be the type of the tree-wall $\T_n$.  Then by \fullref{c:gallery-connected} above, there is a chamber $\phi_n$ of $Y$ which is epicormic at $\T_n$ and in the same gallery-connected component of $Y - \T_n \cap Y$ as $\phi_0$, such that for some chamber $\phi'_n$ which is $s_n$--adjacent to $\phi_n$, $\phi'_n$ is also in $Y$.
Now, as $\Gamma$ is a non-cocompact lattice, the orders of the $\Gamma$--stabilizers of the chambers in $Y$ are unbounded.  Hence the tree-wall $\T_n$ and chambers $\phi_n$ and $\phi'_n$ may be chosen so that $|\Stab_{\Gamma}(\phi_n)| < |\Stab_\Gamma(\phi'_n)|$.
\end{proof}

Let $\phi_n$, $\phi'_n$, $\T_n$, and $\gamma$ be as in \fullref{l:transverse gallery} above and let $s=s_n$ be the type of the tree-wall $\T_n$.  Let $\alpha$ be a gallery in $Y - \T_n \cap Y$ from $\phi_0$ to $\phi_n$.  The chambers $\phi_n$ and $\gamma \cdot \phi_n$ are in two distinct components of $X - \T_n$, since they both contain the $s$--panel $\phi_n \cap \phi'_n \subseteq \T_n$, which is fixed by $\gamma$.    Hence the galleries $\alpha$ and $\gamma \cdot \alpha$ are in two distinct components of $X - \T_n$, and so the chambers $\phi_0$ and $\gamma \cdot \phi_0$ are in two distinct components of $X - \T_n$.
Denote by $X_0$ the component of $X - \T_n$ which contains $\phi_0$, and put $Y_0 = Y \cap X_0$.

\begin{lemma}\label{l:epiy}
Let $\phi$ be a chamber in $X_0$ that is epicormic at $\T_n$.  Then $\phi_Y$ is in $Y_0$ and is epicormic at $\T_n \cap Y$.
\end{lemma}

\begin{proof}
We consider three cases, corresponding to the possibilities for tree-walls in \fullref{c:tree wall trichotomy} above.
\begin{enumerate}
\item If $\T_n$ is reduced to a vertex, there is only one chamber in $X_0$ which is epicormic at $\T_n$, namely $\phi_n$.  Thus $\phi = \phi_n = \phi_Y$ and we are done.
\item If $\T_n$ is finite but not reduced to a vertex, the result follows by finitely many applications of \fullref{l:projection preserves adjacency} above.
\item  If $\T_n$ is infinite, the result follows by induction, using \fullref{l:projection preserves adjacency} above,  on
\[ k:= \min \{ d_W(\phi,\psi) \mid \psi \mbox{ is a chamber of $Y_0$ epicormic at $\T_n \cap Y$}\}. \proved \] 
\end{enumerate}\end{proof}

\begin{lemma}\label{l:final}  For all $n > 0$, the complex $D(n)$ is not connected.
\end{lemma}

\begin{proof} Fix $n > 0$, and let $\alpha$ be a gallery in $X$  between a chamber in $X_0 \cap \Gamma \phi_0$ and some chamber $\phi$ in $X_0$ that is epicormic at $\T_n$.  Let $m$ be the length of $\alpha$.

By \fullref{l:projection preserves adjacency} and \fullref{l:epiy} above, the gallery $\alpha$ projects to a gallery $\beta$ in $Y$ between $\phi _0$ and a chamber $\phi_Y$ that is epicormic at $\T_n \cap Y$.  The gallery $\beta$ in $Y$ has length at most $m$.

It follows from \eqref{i:distance to T_n} of \fullref{l:transverse gallery} above that the gallery $\beta$ in $Y$ has length greater than $n$. Therefore $m > n$.  Hence the gallery-connected component of $D(n)$ that contains $\phi _0$ is contained in $X_0$.  As the chamber $\gamma \cdot \phi_0$ is not in $X_0$, it follows that the complex $D(n)$ is not connected. \end{proof}

 This completes the proof, as $\Gamma$ is finitely generated if and only if $D(n)$ is connected for some $n$.

\end{document}